\documentclass{amsart}
\usepackage{latexsym,amsxtra,amscd,ifthen}
\usepackage{amsfonts}
\usepackage{verbatim}
\usepackage{amsmath}
\usepackage{amsthm}
\usepackage{amssymb}
\usepackage{url}

\theoremstyle{plain}
\newtheorem{maintheorem}{Theorem}

\numberwithin{equation}{section}

\newtheorem{theorem}[equation]{Theorem}
\newtheorem{lemma}[equation]{Lemma}
\newtheorem{proposition}[equation]{Proposition}
\newtheorem{corollary}[equation]{Corollary}

\theoremstyle{definition}

\newtheorem{example}[equation]{Example}
\newtheorem{remark}[equation]{Remark}
\newtheorem{question}[equation]{Question}

\newcommand{\cwlt}{(\textup{cwlt})}
\newcommand{\Q}{\mathbb{Q}}
\newcommand{\Z}{\mathbb{Z}}

\DeclareMathOperator{\ch}{char}

\DeclareMathOperator{\Aut}{Aut}
\DeclareMathOperator{\gr}{gr}
\DeclareMathOperator{\hdet}{hdet}
\DeclareMathOperator{\GKdim}{GKdim}
\DeclareMathOperator{\im}{im}
\DeclareMathOperator{\af}{af}
\newcommand{\VA}{V_{n}({\mathcal A})}
\newcommand{\VAM}{V_{n}({\mathcal A}')}

\begin{document}

\title[Invariant theory for quantum Weyl algebras]
{Invariant theory for quantum Weyl
algebras under finite group action}

\author{S. Ceken, J. H. Palmieri, Y.-H. Wang and J. J. Zhang}

\address{Ceken: Department of Mathematics, Akdeniz University, 07058 Antalya,
Turkey}

\email{cekensecil@gmail.com}

\address{Palmieri: Department of Mathematics, Box 354350,
University of Washington, Seattle, Washington 98195, USA}

\email{palmieri@math.washington.edu}

\address{Wang: School of Mathematics,
Shanghai University of Finance and
Economics, Shanghai 200433, China}

\email{yhw@mail.shufe.edu.cn}

\address{Zhang: Department of Mathematics, Box 354350,
University of Washington, Seattle, Washington 98195, USA}

\email{zhang@math.washington.edu}

\begin{abstract}
We study the invariant theory of a class of quantum Weyl algebras
under group actions and prove that the fixed subrings are always
Gorenstein. We also verify the Tits alternative for
the automorphism groups of these quantum Weyl algebras.
\end{abstract}

\subjclass[2000]{Primary 16W20, 16E65}


\keywords{automorphism group, quantum Weyl algebra, Artin-Schelter
Gorenstein property, free subgroup}


\maketitle


\section*{Introduction}
\label{yysec0}

Fix a field $k$. For $n\geq 2$, let $W_n$ be the
\emph{$(-1)$-quantum Weyl algebra}: this is the $k$-algebra generated by
$x_1,\dots, x_n$ subject to the relations $x_i x_j+x_j x_i=1$ for all
$i\neq j$.

\begin{maintheorem}
\label{yythm1}
Assume that $\ch k=0$. Let $n$ be an even integer 
$\geq 4$ and let $G$ be a group acting on $W_n$. Then the fixed subring 
$W_n^G$ under the $G$-action is filtered Artin-Schelter Gorenstein.
\end{maintheorem}

The above theorem was announced in \cite[Theorem 2]{CPWZ1} without
proof. The first purpose of this paper is to provide a proof of
this, using some earlier results from noncommutative invariant theory
\cite{JZ, KKZ1}. Secondly, we discuss the automorphism group of $W_n$ 
when $n\geq 3$ is odd [Theorem~\ref{yythm2}].  We also want to correct 
some small errors in \cite{CPWZ1}: see Remarks \ref{yyrem1.5}(2) and 
\ref{yyrem2.7}. Finally, we give a criterion for an isomorphism question 
for a class of $(-1)$-quantum Weyl algebras [Theorem \ref{yythm3}].

One aspect of invariant theory is to study homological properties 
of fixed subrings (also called invariant subrings) under group 
actions. When $A$ is regular (or has finite global dimension), 
the fixed subring $A^G$ under the action of a finite group $G$ is 
Cohen-Macaulay when $\ch k$ does not divide the order of $G$. 
One interesting question is \emph{when a fixed subring $A^G$ 
is Gorenstein}. The famous Watanabe theorem \cite{Wa} answers such a question 
for commutative polynomial rings. Theorem \ref{yythm1} above provides 
a solution for $W_n$ when $n$ is even. However, this question is open 
for $W_n^G$ when $n$ is odd (even when $n=3$ and $G$ is finite). Another 
interesting question is \emph{when a fixed subring $A^G$ is regular}. 
The classical result of Shephard-Todd-Chevalley \cite{ST, Ch} answers 
this question for commutative polynomial rings. The question is open 
for $W_n$ for all $n\geq 3$, though we have an easy partial result: see 
Proposition \ref{yypro1.8}. A recent survey on invariant theory of 
Artin-Schelter regular algebras is given by Kirkman in \cite{Ki}. 

Another aspect of invariant theory is to study the structure of the 
automorphism group $\Aut (A)$ of an algebra $A$. There is a long history 
and an extensive study of the automorphism groups of algebras. Determining 
the full automorphism group of an algebra is generally a notoriously 
difficult problem. Recently, significant progress has been made in 
finding the full automorphism groups of some noncommutative algebras. For 
example, during the last two years, Yakimov has proved the 
Andruskiewitsch-Dumas conjecture and the Launois-Lenagan conjecture 
by using a rigidity theorem for quantum tori \cite{Y1, Y2}. The 
automorphism groups of generalized or quantum Weyl algebras have been 
studied by several researchers \cite{AD, BJ, SAV}. The authors used 
the discriminant method to determine the automorphism groups of some 
noncommutative algebras \cite{CPWZ1, CPWZ2}. When $n$ is even, 
$\Aut(W_n)$ was worked out by the authors in \cite[Theorem 1]{CPWZ1}. 
Unfortunately, we have not been able to determine the full automorphism group 
of $W_n$ when $n$ is odd \cite[Example 5.10 and Question 5.13]{CPWZ1}. 

The second theorem concerns $\Aut(W_n)$ when $n$ is odd. 
We will consider a slightly more general setting. Let ${\mathcal A}:=
\{a_{ij}\in k \mid 1\leq i<j\leq n\}$ be a set of scalars. 
Define a modified $(-1)$-quantum Weyl algebra $\VA$ 
\cite[Section 4]{CPWZ1} to be the $k$-algebra generated by 
$\{x_1,\dots,x_n\}$ subject to the relations
\[
x_ix_j+x_jx_i=a_{ij}, \quad \forall \; 1\leq i<j\leq n.
\]
As a special case, note that if $a_{ij}=1$ for all $i<j$, then $\VA=W_n$.
Another special case is when $a_{ij}=0$ for all $i<j$, in which case 
$\VA$ is just the skew polynomial ring 
$k_{-1}[x_1,\dots,x_n]$.

\begin{maintheorem}
\label{yythm2}
Let $n\geq 3$ be an odd integer. Suppose $\ch k$ does not 
divide $(n-1)!$. Then $\Aut(\VA)$ contains a free group on two 
generators. As a consequence, it contains a free group on countably 
many generators.
\end{maintheorem}

Note that \cite[Theorem 1]{CPWZ1} implies that, when $n$ is even,
$\Aut(W_n)$ is finite (and so virtually solvable). A more general 
statement for $\Aut(\VA)$ is Lemma 1.1. Combining Theorem \ref{yythm2} 
with Lemma 1.1(2,3), we obtain that in characteristic 0, $\Aut(\VA)$ 
either is virtually solvable or contains a free subgroup of rank two. 
This can be viewed as a version of the Tits alternative \cite{Ti} for 
the automorphism groups of $\VA$. The original Tits alternative states 
that every finitely generated linear group is either virtually solvable 
or contains a free subgroup of rank two. We also have a version of the 
Tits alternative for the class of the automorphism groups of skew 
polynomial rings in a separate paper \cite{CPWZ3}.

In Section \ref{yysec3}, we use the discriminant to prove the following 
criterion for when two $\VA$s are isomorphic (in the case when $n$ 
is even). For simplicity, let $a_{ij}=a_{ji}$ if $i>j$.

\begin{maintheorem}
\label{yythm3}
Suppose $\ch k\neq 2$ and let $n$ be an even integer. Let ${\mathcal A}':=
\{a'_{ij}\mid 1\leq i<j\leq n\}$ be another set of scalars in $k$. 
Then $\VA \cong \VAM$ if and 
only if there are a permutation $\sigma\in S_n$ and nonzero scalars
$\lambda_i$ for $i=1,\dots,n$ such that $a'_{ij}= \lambda_i \lambda_j
a_{\sigma(i)\sigma(j)}$ for all $i$ and $j$.
\end{maintheorem}

As a consequence of Theorem \ref{yythm3}, when $n\geq 4$ is even,
there are infinitely many non-isomorphic $\VA$s.

In Section 4, we give some examples of $\Aut(\VA)$ for $n=4,6$ and 
list some questions about $\VA$.

\section{Proof of Theorem \ref{yythm1}}
\label{xxsec1}
Throughout let $k$ be a base commutative domain. Modules, vector spaces,
algebras, and morphisms are over $k$. In this section we further assume
that $k$ is a field. 

As above, for each $n\geq 2$, let $\VA$ denote the 
algebra generated by $\{x_i\}_{i=1}^n$ subject to the relations 
$x_i x_j+x_j x_i=a_{ij}$ for all $i\neq j$. This is a filtered 
Artin-Schelter regular algebra in the following sense.
We refer to \cite[Definition 1.2]{KKZ2} for the definition of
(connected graded) Artin-Schelter regularity. 
Let $F:=\{F_n\subseteq A \mid n \geq 0\}$ be an increasing 
filtration on an algebra $A$ satisfying
\begin{enumerate}
\item[(a)]
$F_0=k$,
\item[(b)]
$F_{n}F_{m}\subseteq F_{n+m}$ for all $n,m\geq 0$,
\item[(c)]
$A=\bigcup_{n} F_n$.
\end{enumerate}
The associated graded algebra with respect to $F$, denoted by $\gr_F A$, 
is defined to be $\gr_F A=\bigoplus_{i=0}^{\infty} F_i/F_{i-1}$. This
is a connected graded algebra by condition (a). An algebra $A$ is called 
\emph{filtered Artin-Schelter regular} 
(resp.~\emph{filtered Artin-Schelter Gorenstein}) if there is a
filtration $F$ such that $\gr_F A$ is Artin-Schelter regular
(resp.~Artin-Schelter Gorenstein). In our case, we filter 
$\VA$ by setting $F_0=k$, $F_1=k+\sum_{s=1}^n k x_s$, and 
$F_i=(F_1)^i$ for all $i\geq 2$. Then it is easy to see that $\gr_F 
\VA \cong k_{-1}[x_1,\dots,x_n] (=V_n(\{0\}))$, the skew 
polynomial ring generated by $\{x_i\}_{i=1}^n$ subject to the relations 
$x_jx_i=-x_i x_j$ for all $i<j$. Since $k_{-1}[x_1,\dots,x_n]$
is Artin-Schelter regular, $\VA$ is filtered 
Artin-Schelter regular. The symmetric group $S_n$ acts on the set 
$\{x_i\}_{i=1}^n$ naturally. It is easy to see that this action extends 
to an $S_n$-action on the algebras $\VA$ and $k_{-1}[x_1,\dots,x_n]$ 
uniquely; note that the invariant theory of $S_n$-actions on 
$k_{-1}[x_1,\dots,x_n]$ was studied in \cite{KKZ1}. 

The group of all affine automorphisms of $\VA$, denoted 
by $\Aut_{\af}(\VA)$, can be worked out by using easy 
combinatorics \cite[Lemma 4.3]{CPWZ1}. Let $G({\mathcal A})$ be the 
set of automorphisms of $\VA$ of the form
\begin{equation}
\label{E1.0.1}\tag{E1.0.1}
g: \quad x_i\to r_i x_{\sigma(i)} \; \forall \; i
\end{equation}
where $\sigma\in
S_n$ and $r_i\in k^\times = k \setminus \{0\}$.  It is easy to see
that $g$ is an automorphism of $\VA$ if and only if
\begin{equation}
\label{E1.0.2}\tag{E1.0.2}
a_{ij}=r_i r_j a_{\sigma(i)\sigma(j)}
\end{equation} 
for all $i\neq j$. Such an automorphism $g$ of $\VA$ is denoted by 
$g(\sigma,\{r_1,\dots,r_n\})$. The following lemma is easy.
We will use another obvious automorphism of $\VA$, which sends $x_i$ 
to $-x_i$ for all $i=1,\dots,n$. We denote this automorphism by $-1$,
while the identity is just denoted by $1$. 

\begin{lemma}
\label{yylem1.1} 
Suppose that $n\geq 2$.
\begin{enumerate}
\item[(1)]
\cite[Lemma 4.3]{CPWZ1}
$G({\mathcal A})$ is a group, and $\Aut_{\af}(\VA)=G({\mathcal A})$.
\item[(2)]
There is a short exact sequence
\[
1\to Z \to G({\mathcal A})\to S_n
\]
where $Z$ is a subgroup of $(k^\times)^n$. As a consequence, $G({\mathcal A})$
is virtually abelian {\rm{(}}so virtually solvable{\rm{)}}.
\item[(3)]
\cite[Theorem 4.9(3)]{CPWZ1}
If $n$ is even, then $\Aut(\VA)=
\Aut_{\af}(\VA)=G({\mathcal A})$.
\end{enumerate}
In parts {\rm{(}}4--5{\rm{)}} we further assume that $a_{ij}\neq 0$ for all $i<j$.
\begin{enumerate}
\item[(4)]
If $n\geq 3$, there is a short exact sequence
\[
1\to \{\pm 1\} \to G({\mathcal A})\to S_n.
\]
\item[(5)]
If $n\geq 4$ is even, then $\Aut(\VA)=G({\mathcal A})$ 
is finite. 
\end{enumerate}
\end{lemma}

\begin{proof} 
(1) By a direct computation, $G({\mathcal A})\subseteq
\Aut_{\af}(\VA)$.

Now fix $g\in \Aut_{\af}(\VA)$. By \cite[Lemma 4.3]{CPWZ1},
there is a permutation $\sigma\in S_n$ and scalars $r_i\in k^\times$
such that $g(x_i)=r_i x_{\sigma(i)}$ for all $i$. Applying $g$
to the relations of $\VA$, one obtains that
$a_{ij}=r_i r_j a_{\sigma(i)\sigma(j)}$ for all $i,j$. So $g\in G({\mathcal A})$.
Thus $\Aut_{\af}(\VA)\subseteq G({\mathcal A})$.

(2) The map from $G({\mathcal A})$ to $S_n$ takes
$g(\sigma, \{r_1,\dots,r_n\})$ to $\sigma$. 
So the kernel $Z$ of this map consists of all $g(Id, \{r_1,\dots,r_n\})$
where $a_{ij}=r_i r_ja_{ij}$ for all $i\neq j$. Since each $r_i$ is in 
$k^\times$, $Z$ is a subgroup of $(k^\times)^n$. Since $Z$ is abelian,
$G({\mathcal A})$ is virtually abelian by definition.

(3) By \cite[Theorem 4.9(3)]{CPWZ1}, $\Aut(\VA)=
\Aut_{\af}(\VA)$. The assertion follows from part (1).

(4) By the proof of part (2), $Z$ consists of all $g(Id, \{r_1,\dots,r_n\})$
where $\{r_i\}_{i=1}^n $ satisfy $a_{ij}=r_i r_ja_{ij}$ for all $i\neq j$. 
Since $a_{ij}\neq 0$, we have $r_ir_j=1$ for all $i<j$. Since $n\geq 3$, we 
obtain that $r_i=r_j$ for all $i<j$ and so $r_i=1$ or $r_i=-1$. 

(5) This follows from parts (3,4).
\end{proof}

Let $G_{1}({\mathcal A})$ denote the subgroup of $G({\mathcal A})$
consisting of all $g(\sigma, \{r_1,\dots,r_n\})$ such that
$\prod_{i=1}^n r_i=1$.  If ${\mathcal A}=\{1\}_{i<j}$ (i.e., the
$W_{n}$ case) and $n\geq 1$ is even, then $G({\mathcal A})=
G_{1}({\mathcal A})$ by \cite[Theorem 1]{CPWZ1}. We give an example
for which $G({\mathcal A})\neq G_{1}({\mathcal A})$ when $n=4$.

\begin{example}
\label{yyex1.2} Let $n=4$ and $(a_{ij})_{4\times 4}=
\begin{pmatrix} * & 1 & -1& 1\\
1& * & 1& -a^2 \\-1& 1& * & -a\\
1& -a^2& -a& *\end{pmatrix}$ where $a^3=-1$ and $a\neq -1$. It is easy to 
check that $g((123), \{1,1,-1,a\})$ is in $G({\mathcal A})$, but not in 
$G_1({\mathcal A})$. 
\end{example}

We refer to \cite[Definition 0.6]{JZ} for the definitions of
Auslander-Gorenstein and GKdim-Macaulay.

\begin{theorem}
\label{yythm1.3}
Let $G$ be a finite subgroup of $G_{1}({\mathcal A})$. Suppose that 
$\ch k$ does not divide $|G|$. Then the fixed subring $\VA^G$ 
is filtered Artin-Schelter Gorenstein, Auslander-Gorenstein and 
$\GKdim$-Macaulay.
\end{theorem}

\begin{proof} First of all, we note that both $k_{-1}[x_1,\dots,x_n]$ and 
$\VA$ are filtered Artin-Schelter Gorenstein, Auslander-Gorenstein and 
$\GKdim$-Macaulay. We claim that, for each $g\in G\subset 
G_{1}({\mathcal A})$, its homological determinant, denoted by $\hdet_{\VA} g$, 
is 1. If the claim holds, then the assertion follows from 
\cite[Theorem 3.5]{JZ}. 

Now we prove the claim. By definition \cite[p. 677]{JZ}, 
$\hdet_{\VA} g=\hdet _{\gr \VA} g$. To show the claim, we consider
the algebra $\gr \VA\cong k_{-1}[x_1,\dots,x_n]:=B$ and identify 
$G$ with a subgroup of $\Aut(B)$ in the natural way. For any 
element $g=g(\sigma, \{r_1,\dots,r_n\})\in G$, there is a decomposition
\[
g(\sigma, \{r_1,\dots,r_n\})=g(\sigma,\{1,\dots,1\})\circ 
g(Id, \{r_1,\dots,r_n\})
\]
as in $\Aut(B)$. Then the homological determinant 
of $g(\sigma,\{1,\dots,1\})$ is 1 by the proof of \cite[Theorem 1.5]{KKZ1}. 
The homological determinant of $g(Id,\{r_1,\dots,r_n\})$ 
is $\prod_{i=1}^n r_i$, which is 1 by the definition of $G_1({\mathcal A})$.
Combining these, we see that $\hdet_{B} g=1$ for all 
$g\in G$. Thus we have proved the claim, and therefore the theorem.
\end{proof}

Now we are ready to prove Theorem \ref{yythm1}.

\begin{theorem}
\label{yythm1.4}
Let $n$ be an even integer $\geq 4$ and $G$ a group acting on
$W_n$ faithfully. Suppose that $\ch k$ does not divide
$2|G|$. Then the fixed subring $W_n^G$ is filtered 
Artin-Schelter Gorenstein, Auslander-Gorenstein, and $\GKdim$-Macaulay.
\end{theorem}

\begin{proof} Since $n$ is even and $\ch k\neq 2$, by \cite[Theorem 1]{CPWZ1}, 
$\Aut(W_n)=S_n\times \{\pm 1\}$. So $G$ is a subgroup of 
$S_n\times \{\pm 1\}$, which is $G_1(\{1\}_{i<j})$ when $n$ is even. The 
assertion follows from Theorem \ref{yythm1.3}.
\end{proof}

\begin{remark}
\label{yyrem1.5}
\begin{enumerate}
\item[(1)]
Theorem \ref{yythm1} is clearly a consequence of Theorem 
\ref{yythm1.4}: if the group $G$ acts on $W_n$, then $W_n^G =
W_n^{\im G}$, where $\im G$ is the image of $G$ in $\Aut W_n$. Note
that $\im G$ acts faithfully.
\item[(2)]
The hypothesis that ``$\ch k$ does not divide
$2|G|$'' is missing from \cite[Theorem 2]{CPWZ1}. 
We do not know whether $W_{n}^{G}$
is filtered Artin-Schelter Gorenstein when $\ch k$ 
divides $2|G|$.
\item[(3)] 
By \cite[Example 3.4(2) for $D_4$]{CWWZ}, the fixed subring
$W_2^{S_2\times \{\pm 1\}}$ is a commutative Gorenstein algebra with
isolated singularities.  So we expect that, in general, $W_n^G$ is not
regular. (Note that \cite[Example 3.4(2)]{CWWZ} was stated
in terms of the ring $R = k\langle u,v\rangle / (u^{2}+v^{2}-1)$,
but this is isomorphic to $W_2$.)
\end{enumerate}
\end{remark}

For completeness we have the following assertion for $n=2$; see also
\cite[Theorem 0.1]{CWWZ}.

\begin{proposition}
\label{yypro1.6}
Let $G$ be a finite group acting on $W_2$ faithfully. Suppose that 
$\ch k$ does not divide $2|G|$. Then the 
fixed subring $W_2^G$ is filtered Artin-Schelter Gorenstein, 
Auslander-Gorenstein, and $\GKdim$-Macaulay.
\end{proposition}

\begin{proof} By \cite{AD}, $\Aut(W_2)=S_2\ltimes k^\times$. 
So $G$ is a finite subgroup of $S_2\ltimes k^\times$. Note that 
the $G$-action preserves the given filtration of $W_2$. By 
\cite[Lemma 2.7]{CWWZ}, the homological determinant of the $G$-action 
is trivial (namely, $\hdet_{W_2} g=1$ for all $g\in G$). The assertion 
follows from \cite[Theorem 3.5]{JZ}.
\end{proof}

By \cite[Theorem 1.5(2)]{KKZ1}, if $G$ is a subgroup of $S_n$ other
than $\{1\}$, then $k_{-1}[x_1,\dots,x_n]^G$ is not regular: its
global dimension is infinite. When $G$ is a special subgroup of $S_n$,
then $k_{-1}[x_1,\dots,x_n]^G$ is a classical complete intersection
\cite[Theorem 3.12]{KKZ1}. Some version of this should be true for
$W_n$. Here we only give one simple example.

\begin{example}
\label{yyex1.7}
Some of the arguments in this example were motivated by 
\cite[Example 3.1]{KKZ1}.

Let $S_2$ act on $W_2$ in the obvious way. Then $A:=W_2^{S_{2}}$ is a 
complete intersection, in the sense that it is isomorphic to $B/(f)$
for some regular algebra $B$ and some normal nonzerodivisor $f\in B$.
To see this, we note that $A$ is generated by $u=x_1+x_2$ and 
$v=x_1^3+x_2^3$ satisfying two relations $[u^2,v]=[u,v^2]=0$.
By a straightforward computation, we have
\[
r_3:=2 u^6-3u^3 v -3 v u^3 +4 v^2-5u^4+3u^2+4(uv+vu)=0.
\]
One can check that $A$ is isomorphic to $B/(r_3)$, where
$B$ is the down-up algebra $k\langle u,v\rangle/([u^2,v]=[u,v^2])$.
\end{example}

\begin{proposition}
\label{yypro1.8} Let $G$ be the subgroup $\{\pm 1\}$ of $\Aut(\VA)$.
If $a_{ij}\neq 0$ for some $i<j$, then $\VA^G$ is regular.
\end{proposition}

\begin{proof} Define $\deg x_i=1\in \Z/(2)$. This makes $\VA$
into a $\Z/(2)$-graded algebra. Since the automorphism $-1$
will act on a length $s$ monomial $m=x_{i_{1}} \dots x_{i_{s}}$ by
multiplication by $(-1)^{s}$, such a monomial is in $\VA^G$ if and
only if $s$ is even, which holds if and only if $m$ is in the degree
$0$ part $(\VA)_0$ of $\VA$. Thus $\VA^G = (\VA)_0$. Since
$x_ix_j+x_jx_i=a_{ij}\neq 0$ for some $i<j$, $\VA$ is strongly 
$\Z/(2)$-graded.  It is well known that $\VA$ is regular, 
so $(\VA)_0$ is regular of the same global dimension by a theorem 
of Dade's \cite[Theorem 2.8]{Da} that the categories of
$(\VA)_0$-modules and $\Z/(2)$-graded $\VA$-modules are
equivalent.
\end{proof}

\section{Proof of Theorem \ref{yythm2}}
\label{yysec2}
In this section we assume that $k$ is a commutative domain.

Let $A$ be an algebra generated over $k$ by $x_1,x_2,\dots,x_t$ for some 
$t\geq 3$, and let $R$ be the subalgebra generated by $x_3, \dots,
x_t$. Suppose we have an algebra automorphism of $A$ of 
the following form:
\begin{equation}
\label{E2.0.1}\tag{E2.0.1}
g(x_i)=\begin{cases} x_1+a_0+ a_1 x_2  & i=1,\\
x_i & i\neq 1,\end{cases}
\end{equation}
where $a_0,a_1$ are in the subalgebra $R$, and in particular are
fixed by $g$. The following lemma is easy.

\begin{lemma}
\label{yylem2.1} Let $n$ be an integer. Then 
\[
g^{n}(x_i)=\begin{cases} x_1+na_0+na_1 x_2 & i=1,\\
x_i & i\neq 1.\end{cases}
\]
\end{lemma}

Assume that, when $n$ and $m$ are any elements in $k$, the following 
are algebra isomorphisms of $A$:
\begin{equation}
\label{E2.1.1}\tag{E2.1.1}
g^{n}(x_i)=\begin{cases} x_1+na_0+na_1 x_2 & i=1,\\
x_i & i\neq 1,\end{cases}
\end{equation}
and
\begin{equation}
\label{E2.1.2}\tag{E2.1.2}
h^{m}(x_i)=\begin{cases} x_2+mb_0+mb_1 x_1 & i=2,\\
x_i & i\neq 2.\end{cases}
\end{equation}
This is automatic if $\ch k=0$. 
We abuse notation and denote these by $g^{n}$ and $h^m$, even if $n$ and 
$m$ are not integers. It is easy to check that $g^{n} g^{n'} = g^{n+n'}$.

By symmetry, if we also assume that we have an algebra automorphism
\begin{equation}
\label{E2.1.3}\tag{E2.1.3}
h(x_i)=\begin{cases} x_2+b_0+ b_1 x_1 & i=2,\\
x_i & i\neq 2,\end{cases}
\end{equation}
where $b_{0}, b_{1} \in R$, then for any $m \in k$, we have
\[
h^{m}(x_i)=\begin{cases} x_2+mb_0+mb_1 x_1 & i=2,\\
x_i & i\neq 2.\end{cases}
\]

\begin{lemma}
\label{yylem2.2} For any $m,n \in k$,
\[
h^{m}g^{n}(x_i)=\begin{cases} x_1+(na_0+na_1mb_0)
+na_1 x_2+na_1mb_1x_1& i=1,\\
x_2+mb_0+mb_1 x_1 & i=2,\\
x_i & i\neq 1,2,\end{cases}
\]
and
\[
g^{n}h^{m}(x_i)=\begin{cases} x_1+na_0+na_1 x_2& i=1,\\
x_2+(mb_0+mb_1na_0)+mb_1x_1+mb_1na_1x_2 & i=2,\\
x_i & i\neq 1,2.\end{cases}
\]
\end{lemma}

Let $z$ be a positive integer and $n_1,\dots,n_z, m_1,\dots,m_z$ 
nonzero elements in $k$. 
Define automorphisms
\begin{align*}
\tau_z &= g^{n_z} h^{m_{z-1}} g^{n_{z-1}} \cdots h^{m_1}g^{n_1}, \\
\sigma_z &= h^{m_z} \tau_{z} = h^{m_z} g^{n_z}
h^{m_{z-1}} g^{n_{z-1}} \cdots h^{m_1}g^{n_1}.
\end{align*}

\begin{lemma}
\label{yylem2.3} Retain the above notation. We have
\[
\tau_z(x_1)=d_0+ d_1 x_1+ d_2 x_2+ \left(\prod_{s=1}^{z-1} m_s\right) 
\left(\prod_{s=1}^z n_s \right) a_1 (b_1a_1)^{z-1} x_2
\]
where $d_0,d_1,d_2$ are elements in $R$. Further,
\[
d_1=\sum_{s=0}^{z-1} \alpha_s (a_1b_1)^{s} \quad
{\text{and}}\quad
d_2=\sum_{s=0}^{z-2} \beta_s a_1(b_1a_1)^s
\]
where $\alpha_s,\beta_s$ are elements in the subring of $k$ generated
by $n_1,\dots,n_z,m_1,\dots,m_z$.
\end{lemma}

\begin{proof}
When $z=1$, this is Lemma \ref{yylem2.1}. Now assume that $z>1$.
By definition, $\tau_{z}=g^{n_z}h^{m_{z-1}}\tau_{z-1}$. By a direct 
computation and the induction hypothesis,
\[
\begin{aligned}
\tau_z(x_1)&=g^{n_z}h^{m_{z-1}}\tau_{z-1}(x_1)\\
&=g^{n_z}h^{m_{z-1}} (d'_0+ d'_1 x_1+ d'_2 x_2+ \left(\prod_{s=1}^{z-2} m_s\right)
\left(\prod_{s=1}^{z-1} n_s \right) a_1(b_1a_1)^{z-2} x_2)\\ 
&=d'_0+d'_1(x_1+n_za_0+n_za_1x_2)\\
&\qquad +d'_2(x_2+(m_{z-1}b_0+m_{z-1}b_1n_za_0)
+m_{z-1}b_1x_1+m_{z-1}b_1n_za_1 x_2)\\
&\qquad + \left(\prod_{s=1}^{z-2} m_s\right) 
\left(\prod_{s=1}^{z-1} n_s \right)a_1(b_1a_1)^{z-2} \cdot\\ 
&\qquad\qquad 
(x_2+(m_{z-1}b_0+m_{z-1}b_1n_za_0)
+m_{z-1}b_1x_1+m_{z-1}b_1n_za_1 x_2)\\
&=d_0+ d_1 x_1+ d_2 x_2+ \left(\prod_{s=1}^{z-1} m_s \right) 
\left(\prod_{s=1}^z n_s \right) a_1(b_1a_1)^{z-1} x_2
\end{aligned}
\]
where
\[
\begin{aligned}
d_0&=d'_0+d'_1n_za_0+d'_2(m_{z-1}b_0+m_{z-1}b_1n_za_0)\\
&\qquad 
+ \left(\prod_{s=1}^{z-2} m_s\right) 
\left(\prod_{s=1}^{z-1} n_s \right)a_1(b_1a_1)^{z-2}(m_{z-1}b_0+m_{z-1}b_1n_za_0),
\\
d_1&=d'_1+d'_2m_{z-1}b_1
+ \left(\prod_{s=1}^{z-2} m_s\right) 
\left(\prod_{s=1}^{z-1} n_s \right)a_1(b_1a_1)^{z-2}m_{z-1}b_1,\\
d_2&=d'_1 n_za_1+d'_2+d'_2m_{z-1}b_1n_za_1
+\left(\prod_{s=1}^{z-2} m_s\right) \left(\prod_{s=1}^{z-1} n_s \right)a_1(b_1a_1)^{z-2}.\\
\end{aligned}
\]
The formulas for $d_1$ and $d_2$ follow easily by 
induction.
\end{proof}

In the next proposition, the free product $G*H$ of two groups
$G$ and $H$ will be used. It is well known that every element in
$G*H$ can be expressed uniquely in the form $g_1h_1g_2h_2\cdots g_s h_s$ 
for some $s$ and for some $g_i\in G$, $h_i\in H$ with $h_i\neq 1$ for 
$i\neq s$ and $g_i\neq 1$ for $i\neq 1$. The following lemma 
is well known.

\begin{lemma}
\label{yylem2.4} Let $G$ be a group of order at least 3. Then
$G*G$ contains a free subgroup of rank two.
\end{lemma}

\begin{proof} If $x\in G$ has infinite
order, then $\langle x\rangle * \langle x\rangle$ is a free group
of rank two contained in $G*G$.

If $x$ has order $n\geq 3$, then 
$\langle x\rangle * \langle x\rangle$ contains a free subgroup 
of rank $(n-1)^2$ \cite[Theorem 1]{AP}.

If every element $x$ in $G$ has order 2, then $G$ is abelian and 
$G$ contains $(\Z/(2))^{\oplus 2}$. It is well known
that $(\Z/(2))^{\oplus 2} * (\Z/(2))^{\oplus 2}$ 
contains a free subgroup of rank $(4-1)^2$ \cite[Theorem 1]{AP}. 
\end{proof}

Let $k_0$ be a subring of $k$. 
We say an element $f\in A$ is \emph{transcendental over
$k_0$} if there are no $r_i\in k_0$ with $r_n\neq 0$
such that $\sum_{i=0}^n r_i f^i=0$. 

\begin{proposition}
\label{yypro2.5} Retain the hypotheses and notation as above; in
particular, assume that $g^n$ and $h^m$ from \eqref{E2.1.1} and
\eqref{E2.1.2} are automorphisms of $A$ for all $n,m\in k$.
Assume that $R+Rx_1+Rx_2$ is a free left $R$-module with basis
$\{1,x_1,x_2\}$ and that the element $a_1b_1 \in A$ is transcendental
over a subring $k_0\subseteq k$.
\begin{enumerate}
\item[(1)]
Suppose $m_s$ and $n_s$ are nonzero elements in $k_0$ for all $s$. 
Then neither $\sigma_z$ nor $\tau_z$ is the identity automorphism.
\item[(2)]
$\Aut(A)$ contains the free product $(k_0,+)*(k_0,+)$. As a consequence,
if $k_0\neq \Z/(2)$, then $\Aut(A)$ contains a free subgroup
of rank two.
\end{enumerate}
\end{proposition}

\begin{proof} (1) Since $a_1b_1$ is transcendental over $k_0$, 
$d_2+(\prod_{s=1}^{z-1} m_s) (\prod_{s=1}^z n_s ) a_1(b_1a_1)^{z-1}$
is not zero. Since $R+Rx_1+Rx_2$ is free, $\tau_{z}(x_1)\neq x_1$ 
by Lemma \ref{yylem2.3}.
So $\tau_z$ is not the identity. Since $h^{-m_z}(x_1)=x_1$, we 
deduce that $\tau_z\neq h^{-m_z}$, so
$\sigma_z$ is not the identity.

(2) There is a natural group map $(k_0,+)*(k_0,+) \to \Aut(A)$
sending $m_z*n_z*m_{z-1}*n_{z-1}*\cdots * m_1* n_1$ to
$h^{m_z}g^{n_z} h^{m_{z-1}} g^{n_{z-1}}\cdots h^{m_1}g^{n_1}$.
By part (1) the kernel of this map is $\{1\}$.
So the main assertion follows. The consequence
follows from Lemma \ref{yylem2.4}.
\end{proof}

Here is a version of \cite[Lemma 5.11]{CPWZ1} for the algebra $\VA$.
First recall from \cite[Example 5.10(2)]{CPWZ1},
\[
\Omega(x_1,\dots,x_t):=\sum_{\sigma\in S_t} (-1)^{|\sigma|}
x_{\sigma(1)}\dots x_{\sigma(t)}.
\]

\begin{lemma}
\label{yylem2.6}
For any $t \leq n$, 
consider the element $\Omega_t = \Omega(x_1, \dots, x_t)$ in
$\VA$. Then $x_i \Omega_t =(-1)^{n-1}\Omega_t x_i$ for all $i=1,\dots,
t$.
\end{lemma}

\begin{proof} The proof is similar to that of \cite[Lemma 5.11]{CPWZ1}
and is therefore omitted.
\end{proof}

Now we are ready to prove Theorem \ref{yythm2}.

\begin{proof}[Proof of Theorem \ref{yythm2}]
Let $n\geq 3$ be odd. By Lemma \ref{yylem2.6}, we have two
automorphisms of $\VA$, as in \cite[Example 5.10(2)]{CPWZ1}:
\[
g: x_i\to \begin{cases}
x_1+\Omega(x_2,x_3\dots,x_n) & i=1,\\
x_i & 2\leq i\leq n,\end{cases}
\]
and
\[
h: x_i\to \begin{cases}
x_2+\Omega(x_1,x_3,\dots,x_n) & i=2,\\
x_i & i=1 \; {\text{or}}\; 3\leq i\leq n.\end{cases}
\]
Using the notation in \eqref{E2.0.1} and \eqref{E2.1.3}, 
both $a_1$ and $b_1$ have leading term
equal to $(-1)^{n-2}(n-1)!\prod_{i=3}^{n} x_i$.
It is easy to see that $R+Rx_1+Rx_2$ is a free left $R$-module
with basis $\{1,x_1,x_2\}$, and it is clear that $a_1b_1$ is transcendental
over $k$. Since $k_0=k\neq \Z/(2)$, 
the assertion follows from Proposition \ref{yypro2.5}(2).
\end{proof}

\begin{remark}
\label{yyrem2.7}
It was claimed in \cite[Example 5.10]{CPWZ1} that $\Aut(W_n)$ is not 
affine when $n$ is odd. This is true when $\ch k$ does not
divide $(n-1)!$; it is not known what happens when $\ch k$ divides
$(n-1)!$. The hypothesis that ``$\ch k$ does not
divide $(n-1)!$'' was missing from \cite[Example 5.10]{CPWZ1}.
\end{remark}

We conclude with a corollary of Proposition~\ref{yypro2.5} which deals
with skew polynomial rings. See \cite{CPWZ3} for more results along
these lines.  Let $\{p_{ij}\in k^{\times }:1\leq i<j\leq n\}$ be a set
of parameters. The \emph{skew polynomial ring}, denoted by
$k_{p_{ij}}[x_{1},\dots,x_{n}]$, is defined to be the algebra
generated by $x_{1},\dots, x_{n}$ subject to the relations
$x_{j}x_{i}=p_{ij}x_{i}x_{j}$ for all $i<j$. We set
$p_{ji}=p_{ij}^{-1}$ and $p_{ii}=p_{jj}=1$ for all $i<j$.  For any
$1\leq s\leq n$, let
\[
T_{s}:=\left\{(d_{1}, \dots,\widehat{d_{s}},\dots,d_{n})\in 
\mathbb{N}^{n-1}:\prod\limits_{\substack{ j=1  \\ j\neq
s}}^{n}p_{ij}^{d_{j}}=p_{is} \ \  \forall i\neq s\right\}.
\]
In [CPWZ2, Theorem 3.8], it was proved that if each $p_{ij}$
is a root of unity and $T_{s}=\emptyset $ for all $s$, then
every automorphism of $k_{p_{ij}}[x_{1},\dots,x_{n}]$ is affine. Note also
that if each $p_{ij}$ is a root of unity and some $T_{s}$ is non-empty,
then $T_{s}$ is in fact infinite. If we drop the assumption on the $p_{ij}$'s
and we allow at most one $T_{s}$ to be infinite, then we can still
understand the automorphism group of $k_{p_{ij}}[x_{1},\dots ,x_{n}]$
-- see [CPWZ2,Theorem 4]. The following result concerns
$\Aut(k_{p_{ij}}[x_{1},\dots,x_{n}])$ in another case. 

\begin{corollary}
\label{yycor2.8}
Let $n\geq 3$ be an integer and let $A:=k_{p_{ij}}[x_{1},\dots,x_{n}]$
be a skew polynomial ring. Suppose that $T_{1}$ contains an element of
the form $(1,d_{3},\dots,d_{n})$ and $T_{2}$ contains an element of
the form $(1,d_{3}^{\prime },\dots,d_{n}^{\prime })$. Suppose that at
least one of $d_{3}, \dots, d_{n}, d_{3}', \dots, d_{n}'$ is
nonzero. Then $\Aut(A)$ contains $(k,+)*(k,+)$. If, further, 
$k\neq \Z/(2)$, then $\Aut(A)$ contains a free group on two
generators.
\end{corollary}

\begin{proof}
For $d=(1,d_{3},\dots,d_{n}) \in T_{1}$ and $d^{\prime
}=(1,d_{3}^{\prime },\dots,d_{n}^{\prime }) \in T_{2}$ as in the
statement, define $f_{d}:=x_{2}x_{3}^{d_{3}}\dots x_{n}^{d_{n}}$,
$f_{d^{\prime }}:=x_{1}x_{3}^{d_{3}^{\prime }}\dots
x_{n}^{d_{n}^{\prime }}$. The maps, defined in [CPWZ2, (2.10.1)],
\begin{equation*}
g(f_{d},1):x_{1}\mapsto
\begin{cases}
x_{i} & \text{if }i\neq 1, \\ 
x_{1}+f_{d} & \text{if }i=1,
\end{cases}
\quad \quad 
g(f_{d^{\prime }},2):x_{i}\mapsto
\begin{cases}
x_{i} & \text{if }i\neq 2, \\ 
x_{2}+f_{d^{\prime }} & \text{if }i=2,
\end{cases}
\end{equation*}
extend to algebra automorphisms of $A$. Using the notation
\eqref{E2.0.1} and \eqref{E2.1.3}, $a_{1}=\prod
_{j=3}^{n}p_{j,2}^{d_{j}}x_{3}^{d_{3}} \dots x_{n}^{d_{n}}$ and
$b_{1}=\prod _{j=3}^{n}p_{j,1}^{d_{j}^{\prime }}x_{3}^{d_{3}^{\prime
}} \dots x_{n}^{d_{n}^{\prime }}$. It is clear that $R+Rx_{1}+Rx_{2}$
is a free left $R$-module with basis $\{1,x_{1},x_{2}\}$ and that
$a_{1}b_{1}$ is transcendental over $k$. The assertion follows from
Proposition 2.5(2).
\end{proof}

Note that this applies to the ordinary polynomial ring: if $p_{ij}=1$
for all $i, j$, then each set $T_s$ is all of $\mathbb{N}^{n-1}$.

\section{Proof of Theorem \ref{yythm3}}
\label{yysec3}

In this short section we assume that $k$ is a commutative domain. 
The main idea is to use the discriminant to solve isomorphism questions. Given 
two algebras with similar properties, it is generally very difficult 
to determine whether they are isomorphic. When the discriminant controls
the the automorphism groups of these algebras, it also controls 
the isomorphisms between them.

Given an algebra $A$, we write $d(A/Z(A))$ for the discriminant of $A$
over its center, as defined in \cite{CPWZ1} or \cite{CPWZ2}. The
following lemma is clear.

\begin{lemma}
\label{yylem3.1}
Suppose $A$ and $B$ are PI algebras, and let $g: A\to B$ be an algebra
isomorphism.  Then $g(d(A/Z(A)))=c \; d(B/Z(B))$ for some unit $c\in
Z(B)$.
\end{lemma}

We keep the notation for $\VA$ and $\VAM$ as above, except that it is
convenient to distinguish between their generators, so we use
$\{x_{i}\}$ for the generators of $\VA$ and $\{x_{i}'\}$ for the
generators of $\VAM$.

\begin{proof}[Proof of Theorem \ref{yythm3}] First of all, if there exist
$\sigma\in S_n$ and invertible scalars $\lambda_i$ such that
$a'_{ij}=\lambda_i \lambda_j a_{\sigma(i)\sigma(j)}$, then the map
$g$ determined by $x'_i\mapsto \lambda_i x_{\sigma(i)}$ for $i=1,\dots,n$
extends to an algebra isomorphism $\VAM \to \VA$. 

Conversely, let $g: \VAM\to \VA$ be an algebra isomorphism. 
By \cite[Theorem 4.9]{CPWZ1}, 
\[
d:=d(\VA/Z(\VA))=\left(\prod_{i=1}^n x_i^2\right)^{2^{n-1}}+\cwlt,
\]
where $\cwlt$ (``component-wise less than'') is a linear combination
of terms with degree strictly
less than the degree of $(\prod_{i=1}^n x_i^2)^{2^{n-1}}$.
By symmetry, 
\[
d':=d(\VAM/Z(\VAM))=\left(\prod_{i=1}^n (x'_i)^2\right)^{2^{n-1}}+\cwlt.
\]
By Lemma \ref{yylem3.1}, $g(d')=cd$ for some unit $c\in Z(\VA)=
k[x_1^2,\dots,x_n^2]$. So $c$ is a unit in $k^{\times}$. The equation
$g(d')= cd$ implies that
\begin{equation}
\label{E3.1.1}\tag{E3.1.1}
\left(\prod_{i=1}^n (g(x'_i))^2\right)^{2^{n-1}}+g(\cwlt)=
c \left(\prod_{i=1}^n x_i^2\right)^{2^{n-1}}+\cwlt.
\end{equation}
Note that $g(x'_i)$ has degree at least 1 since it is not in $k$.
So each term in $g(\cwlt)$ has degree strictly less than the
degree of $(\prod_{i=1}^n (g(x'_i))^2)^{2^{n-1}}$. This implies that 
the degree of the left-hand term in \eqref{E3.1.1} is equal to 
the degree of $(\prod_{i=1}^n (g(x'_i))^2)^{2^{n-1}}$. The degree of 
the right-hand term of \eqref{E3.1.1} is $n 2^n$. Therefore 
$\deg g(x'_i)=1$ for all $i$. Thus if we use the standard filtration
on both $\VAM$ and $\VA$, then $g$ preserves the filtration.
Consider the graded algebra isomorphism $\gr g: \gr \VAM\to \gr \VA$,
recalling that $\gr \VA \cong k_{-1}[x_1,\dots,x_n]$. By 
\cite[Lemma 4.3]{CPWZ1}, there are units $r_i\in k$ and a permutation
$\sigma\in S_n$ such that $\gr g (x'_i)=r_i x_{\sigma(i)}$. Going 
back to $g$, there must be $a_i\in k$ such that
$g(x'_i)=r_i x_{\sigma(i)}+a_i$ for all $i$. Using the second 
half of the proof of \cite[Lemma 4.3]{CPWZ1}, one sees that $a_i=0$
for all $i$. The assertion follows.
\end{proof}

\section{Examples and Questions}
\label{yysec4}
In this final section, we give some examples and list some questions 
concerning the algebra $\VA$. For the sake of convenience, we assume 
that $k$ is an algebraically closed field of characteristic zero.

\begin{lemma}
\label{yylem4.1} 
If the order of $\Aut(\VA)$ is $2 n!$, then $n$ is an even integer 
$\geq 4$ and $\VA$ is isomorphic to $W_n$. As a consequence, 
$\Aut(\VA) \cong S_n\times \{\pm 1\}$. 
\end{lemma}

\begin{proof} By Theorem \ref{yythm2}, $n$ is even. 
If $a_{ij}=0$ for all $i<j$, then $\Aut(\VA) \cong S_n \ltimes
(k^{\times})^n$ which does not have order $2 n!$. So
$a_{ij}\neq 0$ for some $i<j$.

If $n=2$, then $a_{12}$ must be nonzero, in which case
$V_{2}(\mathcal{A})$ is isomorphic to $W_{2}$.  By \cite{AD},
$\Aut(W_2)=S_2\ltimes k^{\times}$, so $n\neq 2$. Hence $n$ is an even
integer larger than or equal to $4$.

Given that $a_{ij} \neq 0$ for some $i<j$, we claim that $a_{st}\neq
0$ for all $s < t$.  By Lemma \ref{yylem1.1}, $|\Aut(\VA)| = 2m$,
where $m$ is the order of the image of $\Aut (\VA)$ in $S_{n}$. Since
$|\Aut(\VA)|$ is $2 n!$, the map $\Aut (\VA) \to S_{n}$ must be
surjective, so there is a short exact sequence
\[
1\to \{\pm 1\}\to \Aut(\VA)\to S_n\to 1.
\]
This means that for each $\sigma\in S_n$, there is a corresponding
element $g(\sigma,\{r_i\})\in \Aut(\VA)$. In particular, given $s$ and
$t$ with $1\leq s<t \leq n$,
there is a permutation $\sigma \in S_{n}$ such that $(\sigma(i),
\sigma(j)) = (s,t)$.  By \eqref{E1.0.2}, $a_{ij}=r_ir_j
a_{\sigma(i)\sigma(j)} = r_ir_j a_{st}$.  Since $a_{ij}\neq 0$, we
obtain that $a_{st}\neq 0$.

Since elements in $k$ have square roots in $k$, after replacing
$x_i$ by $b_i x_i$ for $i=1,2,3$ for appropriate $b_1,b_2,b_3$,
we may assume that $a_{12}=a_{13}=a_{23}=1$. For each $i>3$, by
replacing $x_i$ by $a_{1i}^{-1}x_i$, we may assume that $a_{1i}=1$
for all $i>3$. Summarizing, we have $a_{1i}=1$ for all $i>1$ and
$a_{23}=1$.

Next we claim that $a_{ij}=1$ for all $i<j$. First we show that
$a_{2i}=1$ for all $i\neq 2$. Taking an element $g((12),\{r_i\})$ in
$\Aut(\VA)$, \eqref{E1.0.2} implies that
\[
\begin{aligned}
a_{12}&=r_1r_2 a_{21},\\
a_{13}&=r_1r_3 a_{23},\\
a_{23}&=r_2r_3 a_{13}.
\end{aligned}
\]
Since $a_{12}=a_{13}=a_{23}=1$, $r_1=r_2=r_3=\pm 1$. By \eqref{E1.0.2},
\[
a_{3i}=r_3r_i a_{3i} \quad  \, \forall \; i>3.
\]
Hence $r_i=r_3=\pm 1$ for all $i>3$. So $r_i=r_1=\pm 1$ for all $i$. 
By \eqref{E1.0.2} again, $a_{2i}=r_2r_i a_{1i}=1$ for all $i$. By symmetry,
$a_{3i}=1$ for all $i\neq 3$. Secondly, we show that $a_{ij}=1$ for all 
$3<i<j$. Nothing needs to be proved if $n=4$. For $n\geq 6$, it suffices
to show $a_{45}=1$ by symmetry. Taking $g((14),\{r_i\})$ in 
$\Aut(\VA)$, \eqref{E1.0.2} implies that
\[
\begin{aligned}
1=a_{12}&=r_1r_2 a_{42}=r_1r_2,\\
1=a_{13}&=r_1r_3 a_{43}=r_1r_3,\\
1=a_{23}&=r_2r_3 a_{23}=r_2r_3,\\
1=a_{34}&=r_3r_4 a_{31}=r_3r_4.
\end{aligned}
\]
Then $r_1=r_2=r_3=r_4=\pm 1$. By \eqref{E1.0.2},
\[
a_{3i}=r_3r_i a_{3i},\quad  \, \forall \; i>4.
\]
Hence $r_i=r_3=\pm 1$ for all $i>4$. So $r_i=r_1=\pm 1$ for all $i$. 
By \eqref{E1.0.2} again, $a_{45}=r_4r_5 a_{15}=1$, as desired.

The consequence follows from \cite[Theorem 1]{CPWZ1}.
\end{proof}

There are cases when $\Aut(\VA)$ is smaller than
$S_n \times \{\pm 1\}$, as the next example shows. 

\begin{example}
\label{yyex4.2} Let $n=4$.

(1) Let $q$ be transcendental over $\Q\subseteq k$. Let $a_{12}=q$,
$a_{13}=q^{2}$, $a_{14}=q^{4}$, $a_{23}=q^{8}$, $a_{24}=q^{16}$, and
$a_{34}=q^{32}$. For any $\{i_1,\dots,i_4\}=\{j_1,\dots,j_4\}=\{1,\dots,4\}$,
unless $i_s=j_s$ for all $s$, the element
$a_{i_1i_2}a_{i_3i_4}a^{-1}_{j_1j_2}a^{-1}_{j_3j_4}$ is a non-trivial
power of $q$, which is not a root of unity.  Using 
\eqref{E1.0.1}, \eqref{E1.0.2} and the fact that the homological
determinant of $g$ is $r_1r_2r_3r_4$, one can show that
\[
G({\mathcal A})/\{\pm 1\} \cong \{Id, (12)(34),(13)(24),(14)(23)\}
\cong (\Z/(2))^{\oplus 2}.
\]
In general, one can show that when $n=4$ and $a_{ij}\neq 0$
for all $i<j$, then $G({\mathcal A})/\{\pm 1\}$
always contains the subgroup $\{Id, (12)(34),(13)(24),(14)(23)\}$.

(2) Let $q$ be transcendental over $\Q\subseteq k$. Let $a_{12}=q$,
$a_{13}=q^{2}$, $a_{14}=q^{4}$, $a_{23}=q^{8}$, and $a_{24}=a_{34}=0$.
We claim that $\Aut(\VA)\cong S_2\times \{\pm 1\}$.
If $g=g(\sigma, \{r_i\})$ is in $\Aut(\VA)$, then $\sigma$ fixes
$1$ and $4$, as $1$ is the only index so that $a_{1i}\neq 0$ for 
all $i$ and $4$ is the only index so that $a_{i4}=0$ for two different
$i$. If $g$ is neither 1 nor $-1$, $g$ must be of the form
$g((23),\{r_i\})$. In fact, one can check that
$G({\mathcal A})/\{\pm 1\} \cong \{(23)\}\cong S_2$.

(3) 
Let $(a_{ij})_{4\times 4}=
\begin{pmatrix} * & 1 & -1& 1\\
1& * & 1& -a^2 \\-1& 1& * & -a\\
1& -a^2& -a& *\end{pmatrix}$ where $a^3=-1$ and $a\neq -1$. 
This is the same as Example \ref{yyex1.2}. One can show that 
$G({\mathcal A})/\{\pm 1\}\cong A_4.$
\end{example}

\begin{example}
\label{yyex4.3} Let $n=6$.

(1) Suppose that $a_{15}=a_{45}=a_{i6}=0$ for all $i=2,3,4,5$, that
all other $a_{ij}$ are nonzero, and that $a_{12}a_{34}\neq a_{13}a_{24}$.
Then $\Aut(\VA) \cong \{\pm 1\}$. To see this, 
we first note that $6$ is the only index so that there are $4$
different $i$ such that $a_{i6}=0$; thus if $g(\sigma, \{r_i\})$
is in $\Aut(\VA)$, $\sigma$ fixes $6$. Similarly, $\sigma$ fixes
$1, 4$ and $5$. The only possible nontrivial $\sigma$ is $(23)$. In
the case of $g((23), \{r_i\})$, we have
\[
\begin{aligned}
a_{12}&=r_1r_2 a_{13},\\
a_{13}&=r_1r_3 a_{12},\\
a_{23}&=r_2r_3 a_{23},\\
a_{14}&=r_1r_4 a_{14},\\
a_{24}&=r_2r_4 a_{34}.
\end{aligned}
\]
By the first three equations, we have $r_1^2=1$ and $r_2r_3=1$. 
So $r_1=\pm 1$. By the fourth equation, $r_4=r_1$. Then the first and
fifth equations contradict the hypothesis that
$a_{12}a_{34}\neq a_{13}a_{24}$. Therefore the only
$g\in \Aut(\VA)$ is of the form $\pm 1$, so
$\Aut(\VA) \cong \{\pm 1\}$.

(2) Suppose that $a_{12}=0$ and that $a_{ij}=1$ for all $i<j$ except
for $a_{12}$. Then $\Aut(\VA)/\{\pm 1\} \cong S_2\times S_4$ and therefore
$\Aut(\VA) \cong (S_2\times S_4)\times \{\pm 1\}$.
\end{example}

\begin{question}
\label{yyque4.4}
Which finite groups can be realized as $\Aut(\VA)$? For which of those
groups $G$ can one classify the algebras $\VA$ such that $\Aut (\VA)
\cong G$?
\end{question}

Similar to Theorem \ref{yythm1}, we have the following. Recall that
$G_{1}(\mathcal{A})$ is defined just after Lemma \ref{yylem1.1}.

\begin{lemma}
\label{yylem4.5} Suppose there is a short exact sequence
\[
1\to\{\pm 1\}\to \Aut(\VA)\to H\to 1
\]
where $H$ is a subgroup of $S_n$ such that $H=[H,H]$ {\rm{(}}for 
example, $H=A_n$ when $n\geq 6${\rm{)}}. Then
\begin{enumerate}
\item[(1)]
$G({\mathcal A})=G_1({\mathcal A})$.
\item[(2)]
For each subgroup $G\subseteq \Aut(\VA)$, the fixed subring
$\VA^G$ is filtered AS Gorenstein.
\end{enumerate}
\end{lemma}

\begin{proof} Part (2) is a consequence of part (1) and Theorem 
\ref{yythm1.3}, so we only prove part (1). Since $n$ must be even,
the homological determinant $\hdet: \Aut(\VA)\to k^\times$
maps $\pm 1$ to 1. Hence it induces a group homomorphism $\hdet': H\to
k^\times$. Since $H=[H,H]$
and $k^\times$ is abelian, $\hdet'$ is the trivial 
map. So the image of $\hdet$ is $\{1\}$. This is equivalent
to saying that $G({\mathcal A})=G_1({\mathcal A})$.
\end{proof}

Related to the above and Theorem \ref{yythm1}, we have the 
following question.

\begin{question}
\label{yyque4.6}
Classify all $\VA$ (when $n$ is even) such that $\VA^G$ is 
Gorenstein for all subgroups $G\subseteq \Aut(\VA)$.
\end{question}

Theorem \ref{yythm3} suggests the following question.

\begin{question}
\label{yyque4.7}
When $n$ is odd, determine when two $\VA$s are isomorphic.
\end{question}

A related question is the following.

\begin{question}
\label{yyque4.8}
If $\VA$ is Morita equivalent to $\VAM$, is
$\VA$ isomorphic to $\VAM$?
\end{question}

Some Hopf algebra actions on $W_2$ are given in \cite[Example 3.4]{CWWZ}.
It would be very interesting to work out all possible Hopf algebra actions.

\begin{question}
\label{yyque4.9}
Suppose $n$ is even. Classify all finite dimensional Hopf algebras that
can act on $W_n$ inner-faithfully.
\end{question}

%

\subsection*{Acknowledgments} S. Ceken was supported by the Scientific
and Technological Research Council of Turkey (TUBITAK), Science
Fellowships and Grant Programmes Department (Programme no.~2214).
Y.H. Wang was supported by the Natural Science Foundation of China
(grant nos.~10901098 and 11271239), the Scientific Research
Starting Foundation for the Returned Overseas Chinese Scholars,
Ministry of Education of China and the Innovation program of Shanghai 
Municipal Education Commission. J. J. Zhang was supported by the US
National Science Foundation (Nos. DMS-0855743 and DMS-1402863).

\providecommand{\bysame}{\leavevmode\hbox to3em{\hrulefill}\thinspace}
\providecommand{\MR}{\relax\ifhmode\unskip\space\fi MR }
\providecommand{\MRhref}[2]{%

\href{http://www.ams.org/mathscinet-getitem?mr=#1}{#2} }
\providecommand{\href}[2]{#2}


\begin{thebibliography}{10}


\bibitem[AD]{AD}
J. Alev and F. Dumas,
Rigidit{\'e} des plongements des quotients primitifs minimaux de
$U_q(sl(2))$ dans l'alg{\'e}bre quantique de Weyl-Hayashi,
Nagoya Math. J. {\bf 143} (1996), 119--146.


\bibitem[AP]{AP}
M. Anshel and R. Prener, 
On free products of finite abelian groups. 
Proc. Amer. Math. Soc.  {\bf 34} (1972), 343--345. 





\bibitem[BJ]{BJ}
V.V. Bavula and D.A. Jordan,
Isomorphism problems and groups of automorphisms for generalized Weyl
algebras. Trans. Amer. Math. Soc. {\bf 353} (2001), no. 2, 769--794.



\bibitem[CPWZ1]{CPWZ1}
S. Ceken, J. Palmieri, Y.-H. Wang and J.J. Zhang,
The discriminant controls automorphism groups of noncommutative 
algebras, Adv. Math., {\bf 269} (2015), 551-584.


\bibitem[CPWZ2]{CPWZ2}
S. Ceken, J. Palmieri, Y.-H. Wang and J.J. Zhang,
The discriminant criterion and automorphism groups of quantized algebras,
preprint (2014), arXiv:1402.6625 

\bibitem[CPWZ3]{CPWZ3}
S. Ceken, J. Palmieri, Y.-H. Wang and J.J. Zhang,
Tits alternative for automorphisms and derivations, 
in preparation (2015).

\bibitem[CWWZ]{CWWZ}
K. Chan, C. Walton, Y.-H. Wang and J.J. Zhang,
Hopf actions on filtered regular algebras,
J. Algebra, {\bf 397} (2014), 68--90.

\bibitem[Ch]{Ch}
C. Chevalley, Invariants of finite groups generated by
reflections, Amer. J. Math. \textbf{77} (1955), 778-–782.

\bibitem[Da]{Da}
E. Dade, Group-graded rings and modules, Math. Z. \textbf{174} (1980),
241--262.





\bibitem[JZ]{JZ}
N. Jing and J.J. Zhang, 
Gorensteinness of invariant subrings of quantum algebras,
J. Algebra {\bf 221}  (1999),  no. 2, 669--691.

\bibitem[Ki]{Ki}
E. Kirkman,
Invariant theory of Artin-Schelter regular algebras: a survey,
preprint, 2014. 

\bibitem[KKZ1]{KKZ1}
E. Kirkman, J. Kuzmanovich and J.J. Zhang, Invariants of $(-1)$-skew
polynomial rings under permutation representations, 
Contemporary Mathematics, {\bf 623} (2014), 155-192.

\bibitem[KKZ2]{KKZ2}
E. Kirkman, J. Kuzmanovich and J.J. Zhang, 
Invariant theory of finite group actions on down-up algebras, 
Transform. Groups, DOI: 10.1007/S00031-014-9279-4. 





\bibitem[ST]{ST}
G. C. Shephard and J. A. Todd, 
Finite unitary reflection groups, Canadian
J. Math. \textbf{6} (1954), 274-–304


\bibitem[SAV]{SAV}
M. Su{\'a}rez-Alvarez and Q. Vivas,
Automorphisms and isomorphism of quantum generalized Weyl algebras,
preprint (2012), arXiv:1206.4417v1.


\bibitem[Ti]{Ti}
J. Tits, 
Free subgroups in linear groups, 
J. Algebra {\bf 20} (1972), 250--270.


\bibitem[Wa]{Wa}
K. Watanabe, Certain invariant subrings are Gorenstein I, II, 
Osaka J. Math. \textbf{11} (1974) 1--8, 379--388.


\bibitem[Y1]{Y1}
M. Yakimov, The Andruskiewitsch-Dumas conjecture, 
Selecta Math. (N.S.) {\bf 20} (2014),  no. 2, 421--464. 


\bibitem[Y2]{Y2}
M. Yakimov, The Launois-Lenagan conjecture, 
J. Algebra {\bf 392}  (2013), 1--9. 


\end{thebibliography}
\end{document}